\documentclass{article}
\usepackage{amsmath, amssymb, amsthm}
\usepackage{caption} 
\usepackage{hyperref}
\usepackage{setspace} 
\usepackage{tikz}
\usetikzlibrary{decorations.pathreplacing} 

\usepackage{biblatex}
\newtheorem{theorem}{Theorem}

\begin{document}

\title{A Proof Without Words: Triangles in the Triangular Grid}
\author{Peter Kagey}
\date{\today}

\maketitle

\begin{abstract}
  This proof without words demonstrates that there are $\binom{n+2}{4}$
  equilateral triangles in the regular $n$-vertices-per-side triangular grid by
  describing a map from four-element subsets of $\{1,2, \dots, n+2\}$
  into the set of equilateral triangles in this grid.
  Specifically, we illustrate the triangle that corresponds to the
  subset $\{4,5,8,11\}$ under this bijection when $n = 10$.
\end{abstract}

\begin{theorem}
  There are exactly $\binom{n+2}{4}$ equilateral triangles with vertices in a
  triangular region of the regular triangular grid with $n > 2$ vertices on
  each side.
  \end{theorem}
  \begin{proof}
  \noindent
  \begin{center}
    \includegraphics[scale=1]{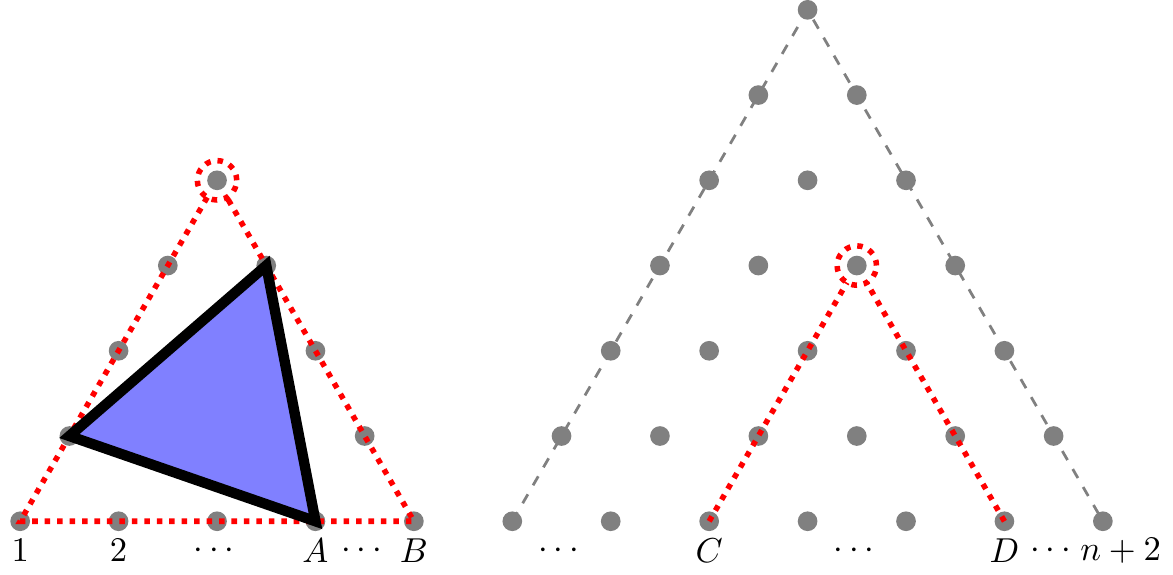}
    \\
    \includegraphics[scale=1]{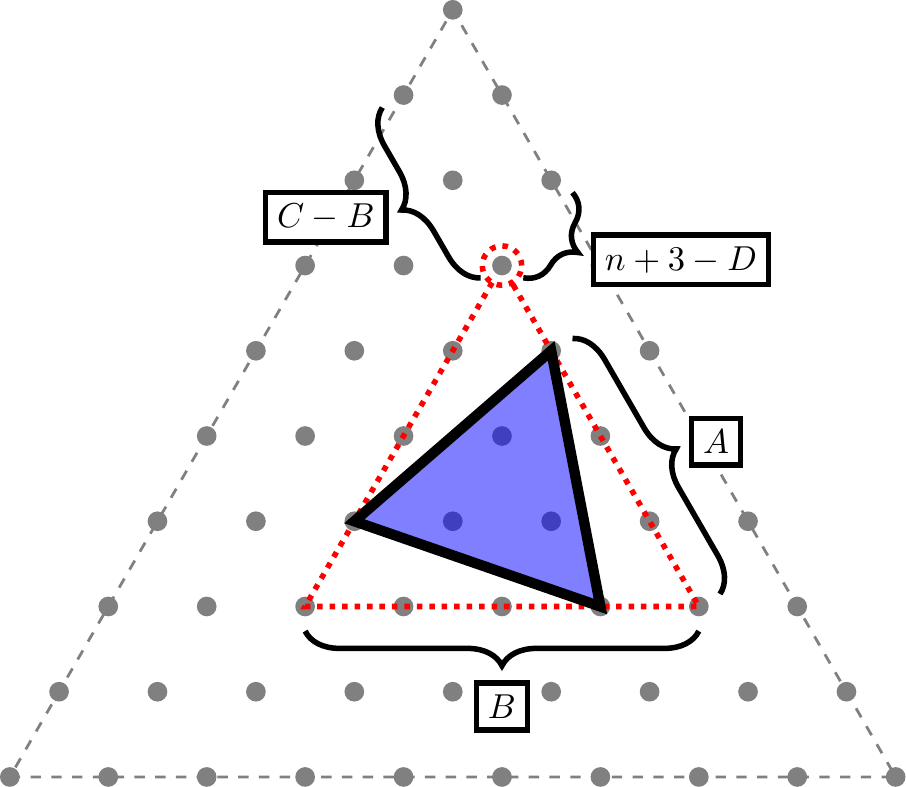}
  \end{center}
  \end{proof}

  \nocite{*}
  \printbibliography
\end{document}